\documentclass[11pt]{amsart}
\usepackage{amsmath}
\usepackage{amssymb}
\usepackage{amsthm}
\usepackage{mathrsfs}
\usepackage{enumerate}

\newtheorem{theorem}{Theorem}[section]

\theoremstyle{remark}

\newcommand{\C}{\mathbb{C}} 		 
\newcommand{\dbar}{\bar{\partial}}	 
\newcommand{\dbars}{\bar{\partial}^{*}}	 
\newcommand{\leb}{\mathcal{L}}       
\newcommand{\dom}{\text{dom}}
\numberwithin{equation}{section}

\begin{document}

\title[Compactness of the $\dbar$-Neumann operator]{Compactness of the $\dbar$-Neumann operator on the intersection of two domains}
\author{Mustafa Ayy\"{u}r\"{u}}
\author{Emil J. Straube}

\address{Rutgers Business School, Master of Quantitative Finance Program,
Newark, NJ, 07102}
\address{Department of Mathematics, Texas A\&M University, College Station, TX, 77843}
\email{mustafa.ayyuru@rutgers.edu}
\email{straube@math.tamu.edu}

\thanks{2000 \emph{Mathematics Subject Classification}: 32W05, 35N15}
\keywords{$\overline{\partial}$--Neumann operator, compactness, pseudoconvex domains, intersections of domains}
\thanks{Research supported in part by NSF grant DMS 0758534.}

\date{August 25, 2014}

\begin{abstract}
Assume that $\Omega_{1}$ and $\Omega_{2}$ are two smooth bounded pseudoconvex domains in $\mathbb{C}^{2}$ that intersect (real) transversely, and that $\Omega_{1} \cap \Omega_{2}$ is a domain (i.e. is connected). If the $\overline{\partial}$--Neumann operators on $\Omega_{1}$ and on $\Omega_{2}$ are compact, then so is the $\overline{\partial}$--Neumann operator on $\Omega_{1} \cap \Omega_{2}$. The corresponding result holds for the $\overline{\partial}$--Neumann operators on $(0,n-1)$--forms on domains in $\mathbb{C}^{n}$. 
\end{abstract}

\maketitle

\begin{center}
\emph{Respectfully dedicated to the memory of M.~Salah Baouendi.}
\end{center}

\section{Introduction}
Let $\Omega$ be a bounded domain in $\C^n$. For $0\le q\le n$, the space of $(0,q)$-forms $u=\sideset{}{_{}^{\prime}}\sum_{|J|=q}u_{J} d\bar{z}_{J}$, where $u_{J}\in \leb^2(\Omega)$ for each strictly increasing $q$-tuple $J$, is denoted by $\leb_{(0,q)}^2(\Omega)$. The inner product on $\mathcal{L}^{2}_{(0,q)}(\Omega)$ is given by
\begin{equation}\label{innerprod}
(u,v)_{\leb_{(0,q)}^2(\Omega)} :=\sideset{}{_{}^{\prime}}\sum_{|J|=q} \int_{\Omega}u_{J}\overline{v_{J}} dV \;,
\end{equation}
where the prime denotes summation over strictly increasing $q$--tuples.
The Cauchy-Riemann operator $\dbar_{q}$ acting on $(0,q)$-forms  is defined as follows:
\begin{equation}\label{dbar}
\dbar_{q}u=\dbar_{q}\left(\sideset{}{_{}^{\prime}}\sum_{|J|=q}u_{J} d\bar{z}_{J}\right)=\sum_{k=1}^{n}\sideset{}{_{}^{\prime}}\sum_{|J|=q}\frac{\partial u_{J}}{\partial \bar{z}_{k}} d\bar{z}_{k}\wedge d\bar{z}_{J}.
\end{equation}
Here, the derivatives are taken in the distributional sense and the domain of $\dbar_{q}$, which we denote by $\dom(\dbar_{q})$, consists of those $(0,q)$-forms with $\dbar_{q}u\in \leb_{(0,q+1)}^2(\Omega)$. Then $\overline{\partial}_{q+1}\overline{\partial}_{q} = 0$; the resulting complex is referred to as 
the $\dbar$ (or Dolbeault)--complex. $\dbar_{q}$ is a linear, densely defined, closed operator, and as such has a Hilbert space adjoint $\dbars_{q}:\leb_{(0,q+1)}^2(\Omega)\to \leb_{(0,q)}^2(\Omega)$. The complex Laplacian is then the unbounded operator $\Box_{q}:=\dbar_{q-1}\dbars_{q-1}+\dbars_{q}\dbar_{q}$, with domain so that the compositions are defined (this imposes a boundary condition not only on a form $u$, but on $\overline{\partial}u$ as well; these are the $\overline{\partial}$--Neumann boundary conditions). It is a deep result of H\"{o}rmander (\cite{MR0179443}) that when $\Omega$ is bounded and pseudoconvex, $\Box_{q}$ is injective and onto, and so has a bounded inverse. This inverse is the $\overline{\partial}$--Neumann operator $N_{q}$. Regularity properties of $N_{q}$ in various function spaces are of great importance, both in several complex variables and in partial differential equations.
We refer the reader to \cite{MR2275654, MR2603659} for details and historical developments.

One of the properties of interest is compactness of $N_{q}$. That is, $N_{q}$ is not just bounded on $\mathcal{L}^{2}_{(0,q)}(\Omega)$, but is compact. Compactness is interesting for a number of reasons: it implies  $\mathcal{L}^{2}$--Sobolev estimates for $N_{q}$ (with all their ramifications) (\cite{MR0181815}), and there are applications to the Fredholm theory of Toeplitz operators (\cite{MR0461588, MR1427681}), to existence or non--existence of Henkin--Ramirez type kernels for solving $\overline{\partial}$ (\cite{MR1842025}), and to certain $C^{*}$ algebras naturally associated to a domain in $\mathbb{C}^{n}$ (\cite{MR1128605}). Details may be found in \cite{MR1912737, MR2603659}, and in their references.

In this note, we address the question of compactness of the $\overline{\partial}$--Neumann operator on the intersection of two domains, given that the $\overline{\partial}$--Neumann operator on each domain is compact. In addition to its intrinsic interest, the question serves as a test for how well compactness of the $\overline{\partial}$--Neumann operator is understood, in particular with respect to identifying `the obstruction to compactness'. Namely, if the obstruction is absent from the boundary of both domains, it ought to be absent from the boundary of the intersection. For example, for a convex domain it is known that $N_{q}$ is compact if and only if the boundary does not contain complex varieties of dimension at least $q$ (\cite{MR1659575, MR1912737, MR2603659}); that is, these varieties in the boundary form the obstruction to compactness. Clearly, if the boundaries of both domains do not contain these varieties, then neither does the boundary of the intersection. Likewise, if 
both boundaries repel `$q$--dimensional analytic structure' in the sense of the potential theoretic sufficient condition known as Property $(P_{q})$ (\cite{MR740870, MR1912737, MR2603659}; see \cite{MR1934357} for a variant), then so does the boundary of the intersection (and the $\overline{\partial}$--Neumann operator on the intersection is compact). However, in general, Property $(P_{q})$ is not known to be equivalent to compactness, so that its failure does not (more precisely, is not known to) constitute an obstruction to compactness. 

We mention that it is straightforward to identify abstractly the obstruction to compactness of the $\overline{\partial}$--Neumann operator via the zero set of the ideal of compactness multipliers introduced in \cite{MR2513533}: the $\overline{\partial}$--Neumann operator is compact if and only if this common zero set is empty. However, so far, it is only possible to identify this set in cases where compactness is understood (i.e. convex domains and Hartogs domains in $\mathbb{C}^{2}$, see \cite{MR2513533}). In particular, we do not understand how the ideal of compactness multipliers on the intersection of two domains arises from the respective ideals on the domains. 

The material presented here is mainly from the PhD dissertation \cite{Mustafa} of the first author, written under the supervision of the second author at Texas A\&M University.

\section{The $\overline{\partial}$--Neumann operators on the intersection of two domains}

Assume now that $\Omega_{1}$and $\Omega_{2}$ are two bounded pseudoconvex domains in $\mathbb{C}^{n}$ whose intersection $\Omega_{1} \cap \Omega_{2}$ is also a domain (i.e. is connected), and whose $\overline{\partial}$--Neumann operators $N_{q}^{\Omega_{1}}$ and $N_{q}^{\Omega_{2}}$ are compact, for some $q$ with $1\leq q \leq n$. Because compactness of $N_{q}$ is a local property (\cite{MR2603659}, Proposition 4.4), one can obtain compactness results for $N_{q}^{\Omega_{1}\cap\Omega_{2}}$ by imposing conditions on $b\Omega_{1}\cap b\Omega_{2}$. In particular, if one assumes that $b\Omega_{1}\cap b\Omega_{2}$ satisfies Property $(P_{q})$ mentioned above, or the variant $(\widetilde{P}_{q})$ from \cite{MR1934357}, then $N_{q}^{\Omega_{1}\cap\Omega_{2}}$ is compact (\cite{Mustafa}, Theorem 4.1.2). We do not pursue this direction here; instead, we focus on the question discussed in the introduction: obtain compactness on the intersection assuming only compactness on the two domains.

The following result, although formulated for domains in $\mathbb{C}^{n}$, is most relevant in dimension $n=2$, as $q=1$ is the case of most interest.

\begin{theorem}\label{main}
Let $\Omega_{1}$ and $\Omega_{2}$ be smooth bounded pseudoconvex domains in $\C^n$ which intersect (real) transversely, and assume that $\Omega_{1}\cap\Omega_{2}$ is a domain (i.e. is connected). If the $\dbar$-Neumann operators $N_{(n-1)}^{\Omega_{1}}$ and $N_{(n-1)}^{\Omega_{2}}$ are compact, then so is $N_{(n-1)}^{\Omega_{1}\cap\Omega_{2}}$.
\end{theorem}

\begin{proof}
For economy of notation, we set $\Omega := \Omega_{1}\cap\Omega_{2}$, and $S:= b\Omega_{1}\cap b\Omega_{2}$. $S$ is a smooth oriented submanifold of $\mathbb{C}^{n}$ of real codimension two, and  $\Omega$ is smooth except at the points of $S$, where it is only Lipschitz. We also omit subscripts from $\overline{\partial}$, as the form level $q$ is clear from the context.

We first note that compactness of $N_{(n-1)}$ is equivalent to the compactness of the canonical solution operators $\dbar^{*}N_{(n-1)}$ and $\dbar^{*}N_{n}$ (see Proposition $4.2$ in \cite{MR2603659}; compare also Lemma 3 in \cite{MR3095048}). $\dbar^{*}N_{n}$ is always compact, because $N_{n}$ maps $W^{-1}_{(0,n)}(\Omega)$ continuously to $W^{1}_{(0,n)}(\Omega)$. That is because for $(0,n)$--forms, the $\overline{\partial}$--Neumann problem reduces to the Dirichlet problem (see for example the discussion following estimate (2.94) on p. 36 of \cite{MR2603659}), and $\Delta: W^{1}_{0}(\Omega) \rightarrow W^{-1}(\Omega)$ is an isomorphism (see for instance Theorem 23.1 in \cite{MR0447753}).
$\dbar^{*}N_{n}$ thus maps $W^{-1}_{(0,n)}(\Omega)$ continuously to $\mathcal{L}^{2}_{(0,n-1)}(\Omega)$, hence is compact as an operator from $\mathcal{L}^{2}_{(0,n)}(\Omega) \rightarrow \mathcal{L}^{2}_{(0,n-1)}(\Omega)$ (since $\mathcal{L}^{2}_{(0,n)}(\Omega)$ embeds compactly into $W^{-1}_{(0,n)}(\Omega)$). Therefore, to show that $N_{(n-1)}$ is compact, it suffices to show that $\dbar^{*}N_{(n-1)}$ is compact. This, in turn, will follow if we can show that there is some compact solution operator for $\dbar$: composing it with the  projection onto $\ker(\dbar)^{\perp}$ (which preserves compactness) gives $\overline{\partial}^{*}N_{(n-1)}$. That is, it suffices 
to find a linear compact operator $T: \leb_{(0,n-1)}^2(\Omega)\cap \ker(\dbar)\to \leb_{(0,n-2)}^2(\Omega)$ such that $\dbar Tu=u$ for all $u\in \ker(\dbar)\cap \leb_{(0,n-1)}^2{(\Omega)}$.

The strategy for constructing $T$ is to write a form $\alpha\in \ker(\dbar_{n-1})\cap \leb_{(0,n-1)}^2{(\Omega)}$ as
\begin{equation}\label{alphaCousin}
\alpha= \beta_{1}|_{\Omega}+ \beta_{2}|_{\Omega} \;, \;\;\beta_{j}\in \ker(\dbar_{n-1})\cap \leb_{(0,n-1)}^2{(\Omega_{j})}, \, j=1,2\; ,
\end{equation}
with
\begin{equation}\label{cousinestimate}
\|\beta_{1}\|_{\mathcal{L}^{2}_{(0,n-1)}(\Omega_{1})} + \|\beta_{2}\|_{\mathcal{L}^{2}_{(0,n-1)}(\Omega_{2})} \lesssim \|\alpha\|_{\mathcal{L}^{2}_{(0,n-1)}(\Omega)} \;,
\end{equation}
and $\beta_{1}$ and $\beta_{2}$ depending linearly on $\alpha$. Then setting $T\alpha := \overline{\partial}^{*}N_{(n-1)}^{\Omega_{1}}\beta_{1} + \overline{\partial}^{*}N_{(n-1)}^{\Omega_{2}}\beta_{2}$ on $\Omega$ gives the desired compact solution operator $T$. We use here that compactness of $N_{(n-1)}^{\Omega_{1}}$ and $N_{(n-1)}^{\Omega_{2}}$ imply compactness of the canonical solution operators $\overline{\partial}^{*}N_{(n-1)}^{\Omega_{1}}$ and $\overline{\partial}^{*}N_{(n-1)}^{\Omega_{2}}$, respectively (see again \cite{MR2603659}, Proposition 4.2).

The situation in \eqref{alphaCousin} is reminiscent of that in a Cousin problem. We proceed accordingly; extra care is needed because we need to control $\mathcal{L}^{2}$--norms.
Because $\Omega_{1}$ and $\Omega_{2}$ intersect transversely, we can choose a partition of unity $\{\varphi, 1-\varphi\}$ of $\Omega_{1}\cup\Omega_{2}$, subordinate to the cover $\{\Omega_{1}, \Omega_{2}\}$, 
with $|\nabla\varphi(z)|\lesssim 1/d_{S}(z)$; here, $d_{S}$ denotes the distance to $S$. We will give details in the appendix (section \ref{app}). Now set
\begin{equation}\label{betadef}
 \widetilde{\beta_{1}} := (1-\varphi)\alpha\;, \;\;\;\widetilde{\beta_{2}} := \varphi\alpha \; .                                                                                               
\end{equation}
We can think of $\widetilde{\beta_{1}}$ and $\widetilde{\beta_{2}}$ as forms in $\mathcal{L}^{2}_{(0,n-1)}(\Omega_{1})$ and $\mathcal{L}^{2}_{(0,n-1)}(\Omega_{2})$, respectively, by setting them zero outside $\Omega$. Of course, the forms need not be $\overline{\partial}$--closed. We have
\begin{equation}\label{dbarbeta}
\overline{\partial}\widetilde{\beta_{1}} = -(\overline{\partial}\varphi\wedge\alpha) \; ,\;\overline{\partial}\widetilde{\beta_{2}} = \overline{\partial}\varphi\wedge\alpha \; ,
\end{equation}
on $\Omega_{1}$ and $\Omega_{2}$ respectively. Now $\overline{\partial}\varphi\wedge\alpha$ is a form on $\Omega_{1}\cup\Omega_{2}$, by setting it equal to zero outside the support of $\nabla\varphi$. If we can write it as $\overline{\partial}\gamma$ on $\Omega_{1}\cup\Omega_{2}$, then setting $\beta_{1} := \widetilde{\beta_{1}}+\gamma$ on $\Omega_{1}$, and $\beta_{2} := \widetilde{\beta_{2}}-\gamma$ on $\Omega_{2}$, produces forms that satisfy \eqref{alphaCousin} (as the two corrections will cancel in the sum). Of course, we also need to preserve the estimates \eqref{cousinestimate} (which are satisfied by $\widetilde{\beta_{1}}$ and $\widetilde{\beta_{2}}$).

Because $\overline{\partial}\varphi\wedge\alpha$ is a $(0,n)$--form, we can solve the equation $\overline{\partial}\gamma = \overline{\partial}\varphi\wedge\alpha$ explicitly on $\Omega_{1}\cup\Omega_{2}$, using again that for $(0,n)$--forms, the $\overline{\partial}$--Neumann problem reduces to the Dirichlet problem for the Laplacian. Define $g$ by $\overline{\partial}\varphi\wedge\alpha = g\,d\overline{z_{1}} \wedge \cdots \wedge d\overline{z_{n}}$. We use again that $\Delta: W^{1}_{0}(\Omega_{1}\cup\Omega_{2}) \rightarrow W^{-1}(\Omega_{1}\cup\Omega_{2})$ is an isomorphism. If we set 
\begin{equation}\label{gamma}
\gamma = \overline{\partial}^{*}\left(-4(\Delta^{-1}g)d\overline{z_{1}}\wedge \cdots \wedge d\overline{z_{n}}\right) \; ,
\end{equation}
then, on $\Omega_{1}\cup\Omega_{2}$,
\begin{equation}\label{gamma1}
\overline{\partial}\gamma = \overline{\partial}\overline{\partial}^{*}\left(-4(\Delta^{-1}g)d\overline{z_{1}}\wedge \cdots \wedge d\overline{z_{n}}\right) = \left(\Delta(\Delta^{-1}g)\right)d\overline{z_{1}}\wedge \cdots \wedge d\overline{z_{n}} = \overline{\partial}\varphi\wedge\alpha \; .
\end{equation}
We have used here that $\Box$ (which equals $\overline{\partial}\overline{\partial}^{*}$ on $(0,n)$--forms) acts diagonally as $(-1/4)\Delta$, see for example \cite{MR2603659}, Lemma 2.11. From \eqref{gamma} we immediately obtain
\begin{equation}\label{gammaest}
\|\gamma\|_{\mathcal{L}^{2}_{(0,n-1)}(\Omega_{1}\cup\Omega_{2})} \lesssim \|\Delta^{-1}g\|_{W^{1}_{0}(\Omega_{1}\cup\Omega_{2})} \lesssim \|g\|_{W^{-1}(\Omega_{1}\cup\Omega_{2})} \simeq \|\overline{\partial}\varphi\wedge\alpha\|_{W^{-1}_{(0,n)}(\Omega_{1}\cup\Omega_{2})} \; . 
\end{equation}
In order to estimate the right hand side of \eqref{gammaest}, we recall that $|\nabla\varphi|\lesssim 1/d_{S} \lesssim 1/d_{b(\Omega_{1}\cup\Omega_{2})}$, where $d_{b(\Omega_{1}\cup\Omega_{2})}$ denotes the distance to the boundary of $\Omega_{1}\cup\Omega_{2}$. This implies that multiplication by a derivative of $\varphi$ maps $W^{1}_{0}(\Omega_{1}\cup\Omega_{2})$ continuously into $\mathcal{L}^{2}(\Omega_{1}\cup\Omega_{2})$ (see for example \cite{MR775683}, Theorem 1.4.4.4; $\Omega_{1}\cup\Omega_{2}$ has a Lipschitz boundary). By duality, this multiplication maps $\mathcal{L}^{2}(\Omega_{1}\cup\Omega_{2})$ continuously into $W^{-1}(\Omega_{1}\cup\Omega_{2})$. As a result, the right hand side of \eqref{gammaest} is dominated by $\|\widetilde{\alpha}\|_{\mathcal{L}^{2}_{(0,n-1)}}(\Omega_{1}\cup\Omega_{2}) = \|\alpha\|_{\mathcal{L}^{2}_{(0,n-1)}(\Omega)}$, where $\widetilde{\alpha} = \alpha$ on $\Omega$, and zero otherwise.

Now we set 
\begin{equation}\label{betadef}
\beta_{1} := (1-\varphi)\alpha + \gamma \;;\;\beta_{2} := \varphi\alpha - \gamma \; .
\end{equation}
Then $\beta_{1}$ and $\beta_{2}$ are $\overline{\partial}$--closed, so that we have \eqref{alphaCousin}. The estimates above imply that \eqref{cousinestimate} also holds. The discussion following \eqref{cousinestimate} shows that the proof of Theorem \ref{main} is now complete.

\end{proof}

\emph{Remark}: One's first tendency would probably be to take the decomposition \eqref{betadef} and apply the compactness estimates on $\Omega_{1}$ and $\Omega_{2}$ to $\widetilde{\beta_{1}}$ and $\widetilde{\beta_{2}}$, respectively. However, derivatives of $\varphi$ blow up at $S$; as a result, $\widetilde{\beta_{1}}$ and $\widetilde{\beta_{2}}$ are not known to be in $\dom(\overline{\partial})\cap \dom(\overline{\partial}^{*})$ on the respective domains. By contrast, our approach above only requires the estimation of $\|\overline{\partial}\varphi\wedge\alpha\|_{W^{-1}_{(0,n)}(\Omega_{1}\cup\Omega_{2})}$, as in \eqref{gammaest}, rather than $\|\overline{\partial}\varphi\wedge\alpha\|_{\mathcal{L}^{2}_{(0,n)}(\Omega_{j})}$, $j=1,2$. This weaker estimate suffices because we can exploit the elliptic gain of the $\overline{\partial}$--Neumann operator on $(0,n)$--forms (which is essentially $\Delta^{-1}: W^{-1}(\Omega_{1}\cup\Omega_{2}) \rightarrow W^{1}_{0}(\Omega_{1}\cup\Omega_{2})$, as explained above) to 
recover the loss that derivatives of $\varphi$ introduce. It is this part of the argument, more than anything else, that confines us to consider only $(0,n-1)$--forms. For example, the fact that $\Omega_{1}\cup\Omega_{2}$ is not pseudoconvex should be less of an issue. In order to prove Theorem \ref{main} via our approach for $(0,q)$--forms, one has to solve $\overline{\partial}$ on $\Omega_{1}\cup\Omega_{2}$ at the level of $(q+1)$--forms. At least in the context of smooth $\overline{\partial}$--cohomology on $\Omega_{1}\cup\Omega_{2}$, the cohomology groups are trivial at levels $q\geq 2$ (this follows from a Mayer--Vietoris sequence argument; see for example \cite{MR0206335}, Proposition (3.7); this reference contains a systematic discussion of `cohomological $q$--completeness').

\section{Appendix}\label{app}

In this section, we show how to construct a partition of unity $\{\varphi, 1-\varphi\}$ on $\Omega_{1}\cup\Omega_{2}$ subordinate to the cover $\{\Omega_{1}, \Omega_{2}\}$, such that we have the estimate $|\nabla\varphi| \lesssim 1/d_{S}$. 

Define the unit vector fields $X$ and $Y$ on $S$ as follows. For $\zeta \in S$, $X(\zeta)$ is the unique unit vector perpendicular to $S$ and tangential to $\Omega_{2}$, such that $X\rho_{1}(\zeta) < 0$ (i.e. $X(\zeta)$ points inside $\Omega_{1}$). The latter is possible because derivatives of $\rho_{1}$ tangential to $\Omega_{2}$ and transverse to $S$ do not vanish,
by transversality of the intersection of $\Omega_{1}$ and $\Omega_{2}$. $Y$ is defined analogously, with the roles of $\Omega_{1}$ and $\Omega_{2}$ interchanged. Then the vector $X+Y$ points inside $\Omega$ at points of $S$. Indeed, for $\zeta \in S$, we have 
\begin{equation}\label{inside}
 (X+Y)\rho_{1}(\zeta) = X\rho_{1}(\zeta) + Y\rho_{1}(\zeta) = X\rho_{1}(\zeta) < 0 \; .
\end{equation}
We have used that $Y\rho_{1} = 0$ ($Y$ is tangential to $\Omega_{1}$). \eqref{inside} says that $X+Y$ points inside $\Omega_{1}$ at points of $S$. Similarly (or by symmetry), this vector also points inside $\Omega_{2}$, hence inside $\Omega$.

Denote by $D_{r}\subset \mathbb{C}$ the disc of radius $r$, centered at $0$. We consider a diffeomorphism $h$ from $S \times D_{r}$, for $r$ sufficiently small, onto a tubular neighborhood $V$ of $S$ (see e.g. \cite{MR0448362}, chapter 4), defined as follows:
\begin{equation}\label{diffeo}
h(\zeta,w) = \zeta + Re(w)X(\zeta) + Im(w)Y(\zeta) \; , \; \zeta\in S\,,\,w\in D_{r} \;.
\end{equation}
By continuity, there is $\alpha >0$ such that the sector of $D_{r}$ where $\pi/4-\alpha \leq arg(w) \leq \pi/4 +\alpha$, less the origin, is mapped into $\Omega$, and the opposite sector is mapped into the complement of $\overline{\Omega_{1}}\cup\overline{\Omega_{2}}$. Here, $\arg(w)$ denotes the branch of the argument with values between $-3\pi/4$ and $5\pi/4$. Choose a function $\sigma\in C^{\infty}(\mathbb{R})$ with $0 \leq \sigma \leq 1$, $\sigma \equiv 1$ on $(-\infty, \pi/4-\alpha]$, and $\sigma \equiv 0$ on $[\pi/4+\alpha, \infty)$. On $V\cap (\Omega_{1}\cup\Omega_{2})$, we 
define $\varphi$ as follows:
\begin{equation}\label{partition}
\varphi(z) = \sigma(\arg(w)) \; ,\; z=h(\zeta,w) \;.
\end{equation}
For points $(\zeta,w)\in h^{-1}(b\Omega_{2}\cap\Omega_{1}\cap V)$, $\arg(w)$ takes values in the sectors $(-3\pi/4+\alpha,\pi/4-\alpha)$, so that for these points, $\sigma(\arg(w)) = 1$ (possibly after shrinking $V$). This is because $S$ and $X(\zeta)$ span the tangent space to $b\Omega_{2}$ at $\zeta$, and $X(\zeta)$ points inside $\Omega_{1}$. Similarly, $\sigma(\arg(w)) = 0$ for points $(\zeta,w)\in h^{-1}(b\Omega_{1}\cap\Omega_{2}\cap V)$.

Arguing geometrically or directly computing one finds that $|\nabla\sigma(\arg(w))| \lesssim 1/|w|$. Because $h$ is a diffeomorphism and $|w|$ is comparable to $d_{S}$, $|\nabla\varphi|$ has the desired upper bound near $S$.

It remains to extend $\varphi$ to $\Omega_{1}\cup\Omega_{2}$. First, we extend $\varphi$ by $0$ into a (small enough) neighborhood in $\Omega_{2}$ of $\Omega_{2}\setminus\Omega_{1}$. Similarly, we extend $\varphi$ by $1$ into a neighborhood in $\Omega_{1}$ of $\Omega_{1}\setminus\Omega_{2}$. Using a suitable cutoff function, $\varphi$ so defined on these neighborhoods and $V$ can be extended from a slightly smaller set via a suitable cutoff function to obtain the function we need on $\Omega_{1}\cup\Omega_{2}$.

\bigskip\bigskip\bigskip

\addcontentsline{toc}{section}{References}
\bibliographystyle{abbrv} \bibliography{biblio}
\end{document}